\documentclass[pdflatex,bst/sn-mathphys-num]{sn-jnl}

\usepackage{graphicx}%
\usepackage{multirow}%
\usepackage{amsmath,amssymb,amsfonts}%
\usepackage{amsthm}%
\usepackage{mathrsfs}%
\usepackage[title]{appendix}%
\usepackage{xcolor}%
\usepackage{textcomp}%
\usepackage{manyfoot}%
\usepackage{booktabs}%
\usepackage{algorithm}%
\usepackage{algorithmicx}%
\usepackage{algpseudocode}%
\usepackage{listings}%
\usepackage{todonotes}
\usepackage{enumerate}

\newtheorem{thm}{Theorem}[section]
\newtheorem{lem}[thm]{Lemma}

\newtheorem{cor}[thm]{Corollary}

\theoremstyle{definition}
\newtheorem{rem}[thm]{Remark}
\newtheorem{defi}[thm]{Definition}

\usepackage{cancel} 
\usepackage{xfrac}  
\usepackage{mathrsfs}
\usepackage{color,xcolor}

\usepackage{float}
\usepackage{tikz}
\usepackage{tikz-3dplot}
\usetikzlibrary{patterns}
\usetikzlibrary{decorations.pathmorphing}
\usepackage{pgfplots}
\pgfplotsset{compat=1.18}
\usepgfplotslibrary{colormaps}
\usepgfplotslibrary{fillbetween}
\usetikzlibrary{3d,calc}
\usepackage{siunitx}
\usepackage[backend=biber
, isbn=false
, doi=false
, url=false
]{biblatex}
\theoremstyle{thmstyleone}%
\theoremstyle{thmstyletwo}%

\theoremstyle{thmstylethree}%
\newtheorem{definition}{Definition}%

\begin{document}

\title[Equivalence of the NEC to variable lower bounds on the timelike Ricci curvature for $C^2$-Lorentzian metrics]{Equivalence of the null energy condition to variable lower bounds on the timelike Ricci curvature for $C^2$-Lorentzian metrics}


\author*[1]{\fnm{Melanie} \sur{Graf}}\email{melanie.graf@uni-hamburg.de}

\author*[1]{\fnm{Yaver} \sur{Gulusoy}}\email{yaver.gulusoy@gmail.com}

\affil[1]{\orgdiv{Department of Mathematics}, \orgname{University of Hamburg}, \orgaddress{\street{Bundesstraße 55},  \postcode{20146} \city{Hamburg}
, \country{Germany}}}


\abstract{The null energy or null convergence condition (NEC) is one of the fundamental assumptions necessary for many celebrated results from Lorentzian Geometry and Mathematical General Relativity. As such there have been several recent efforts to find a good generalization of this condition to the new setting of Lorentzian length spaces or metric measure spacetimes. One important property any such generalization should fulfill is consistency with the classical formulation for a class of spacetimes as large as possible. The purpose of this note is to show that the reformulation of the NEC by McCann as variable lower timelike Ricci curvature bounds in \cite{mccann2024synthetic} remains equivalent to the classical NEC not just for smooth but even for $C^2$-metrics, where McCann's original proof needs to be modified.}


\maketitle

\newpage 
\newcommand{\seqIndex}{m}
\newcommand{\gradRiem}{\operatorname{grad}} 
\newcommand{\horizontalDecomp}[1]{#1^{h}}
\newcommand{\verticalDecomp}[1]{#1^{v}}
\newcommand{\horizontalLift}[1]{hor(#1)}
\newcommand{\verticalLift}[1]{vert(#1)}

\newcommand{\M}{\mathbb{R}^n}
\newcommand{\SasakiMetric}{\widetilde{G}}
\newcommand{\R}{\mathbb{R}}
\renewcommand{\v}{\mathbf{v}}
\renewcommand{\S}{\mathbb{S}}

\section{Introduction}

Curvature or energy conditions of the form $\mathrm{Ric}(\v,\v)\geq 0$ for certain subsets of vectors $\v\in TM$ have been fundamental to many mathematical developments in General Relativity and Lorentzian Geometry since their introduction to the field in the 1960ies. These conditions have found numerous applications, most notably in the formulation of singularity and splitting theorems. As such, finding replacements for the classical energy conditions naturally forms an integral part of the newly developing theory of Lorentzian length spaces or metric measure spacetimes (cf. \cite{cavalletti2022review} and references therein). While timelike Ricci curvature bounds are by now quite well understood even in these settings (starting from the works of \cite{mccanntcd, mondino2022optimal-transport-einstein, CMoriginal}), how to treat the null energy or null convergence condition (NEC), 
\begin{definition}[Classical null energy condition]\label{def:classicNEC}
    A spacetime\footnote{that is a smooth manifold $M$ with a smooth or at least $C^2$  Lorentzian metric $g$ and a time orientation} $(M,g)$ satisfies the classical null energy condition (classical NEC) if
    $$\mathrm{Ric}(\v,\v)\geq 0 \quad \text{ for all }\v\in TM \text{ with } g(\v,\v)=0,$$
\end{definition}
\noindent and null geometry more broadly has remained a more open problem. \\

A better understanding of the NEC specifically is desirable on multiple levels: For one the NEC is weaker and more physically well founded than the strong energy condition. It is also satisfied in scenarios where the strong energy condition is violated even classically, such as the massive scalar field, and while taking into account quantum effects will lead to a violation of the NEC in semiclassical settings there is hope of retaining at least some 'averaged NEC' (cf. the discussions in \cite[p. 723 and 801]{BigSenovillaReview}).

The second reason lies in what is also part of the mathematical difficulty namely that, contrary to timelike Lorentzian geometry, null geometry cannot as directly draw inspiration from Riemannian geometry, where the study of $\mathrm{(R)CD}$ spaces has been a very active field of research since the introduction of $\mathrm{CD}$ spaces by Sturm \cite{sturm2006geometry, sturm2006geometry2} and independently by Lott and Villani \cite{lott2009ricci} in the mid 2000s.\\

Recently two proposals for a NEC in a purely synthetic setting  have been put forward, cf. \cite{mccann2024synthetic} and \cite{CMM2025}. 
Their character is very different: In the first proposal, McCann bases his definition in \cite{mccann2024synthetic} on, essentially, an approximation argument. Since there are timelike vectors arbitrarily close to null vectors and the only way for the quotient $\frac{\mathrm{Ric}(\v,\v)}{|g(\v,\v)|}$ to diverge 
is if $\v$ becomes arbitrarily close to being null, \cite{mccann2024synthetic} first replaces the classical NEC with a variable lower bound on the timelike Ricci curvature,
\begin{definition}[Variable timelike curvature condition]\label{def:RobertNEC}
     A spacetime $(M,g)$ is said to have a variable lower bound on the timelike Ricci curvature if for any compact $Z\subset M$ there exists a constant $C_Z^T\in \R$ such that
    $$\mathrm{Ric}(\v,\v)\geq C_Z^T \, g(\v,\v) \quad \text{ for all }\v\in TM|_Z \text{ with } g(\v,\v)<0,$$
\end{definition}
and then uses this together with one of the established formulations of timelike Ricci curvature bounds in metric measure spacetimes  as the basis for his definition of a NEC for metric measure spacetimes: A (sufficiently nice) metric measure spacetime is said to satisfy a (weak) null energy-dimension condition, $(w)\mathrm{NC}^e_q
(N)$, (with dimension $N>0$) if 
 for any compact $Z\subseteq M$ there exists a constant $C_Z\in \R$ such that the metric measure spacetime $J(Z,Z):=J^+(Z)\cap J^-(Z)$ satisfies the respective timelike curvature dimension condition with curvature bound $C_Z$ and dimension $N>0$.\\ 

While this work will focus on McCann's approach, let us for completeness also briefly discuss the more recent proposal by Cavalletti, Manini and Mondino, \cite{CMM2025}. They define their synthetic null energy condition, $\mathrm{NC}^e(N)$, based on the concavity of the Shannon entropy power functional along optimal transport along a given null hypersurface with prescribed parametrization and measure. This program builds on equivalences established in the smooth setting, cf. their earlier work \cite{CMM2024}  and also \cite{ketterer2024characterization} for  a related approach using optimal transport based on null displacement convexity of the relative Renyi entropy that has been independently  explored by Ketterer. 
 
 As part of Cavalletti, Manini and Mondino's approach, instead of finding a way to incorporate null geometry strictly within the established framework of Lorentzian length spaces or metric measure spacetimes, their central objects of study become what they define to be synthetic null hypersurfaces: Triples $(H,G,\mathfrak{m})$ consisting of a closed achronal set $H$ in a topological causal space, a gauge function $G$ on $H$ encoding the desired parametrization along causal curves in $H$, and a Radon measure $\mathfrak{m}$ on $H$.
  Such synthetic null hypersurfaces can either be studied on their own or potentially as structures within a given larger synthetic Lorentzian space. However, 
 trying to determine synthetic null hypersurfaces within a given synthetic Lorentzian space poses some difficulties: 
 While one may start by considering  
achronal sets 
within a synthetic Lorentzian space, the (equivalence class of) the gauge function and the measure required to turn them into synthetic null hypersurfaces do not always appear canonically arise from the given data.\footnote{One can construct explicit examples of causally well-behaved Lorentzian length spaces where a given achronal boundary inherits two very different gauges and measures turning it into a synthetic null hypersurface in the  sense of Cavalletti, Manini and Mondino. This will be further explored in upcoming work of Sa\'ul Burgos, Leonardo Garc\'ia-Heveling and Melanie Graf.} As a consequence 'satisfying the $\mathrm{NC}^e(N)$ of \cite{CMM2025}' is a priori not a property of the synthetic Lorentzian space itself.

On the other hand, 
their definition very elegantly builds 
on both the geometric meaning of the NEC, 
which largely arises from its consequences on the expansion of families of null geodesic foliating null hypersurfaces, and established techniques from optimal transport. This allows Cavalletti, Manini and Mondino to derive strong consequences, such as a version of the Penrose singularity theorem and of Hawking's area theorem for synthetic null hypersurfaces, which are so far not available for McCann's synthetic NEC.\\

Whichever approach one chooses, any formulation of a NEC for Lorentzian length spaces/metric measure spacetimes should at minimum be equivalent to the classical NEC for a large class of smooth spacetimes -- which both \cite{mccann2024synthetic} and \cite{CMM2025} achieve in principle\footnote{Since \cite{CMM2025}  considers a slightly different setting, their equivalence is not for spacetimes per se, but rather for nice enough null hypersurfaces, such as achronal boundaries, within them, when equipped with their natural geodesic gauge and the corresponding rigged measure.} 
-- and better yet also for less regular spacetimes. Of particular interest are regularities where either the classical NEC is still well-defined ($C^2$-metrics) or other useful generalizations of the NEC, e.g.\ based on almost everywhere ($C^{1,1}_{\mathrm{loc}}$-metrics) or distributional curvature bounds ($C^1$-metrics), are already available and have been successfully used to derive geometric consequences such as a version of Penrose's singularity theorem and related results, cf.\ \cite{graf2020singularity, schinnerl2021note, kunzinger2022hawking}.

The purpose of this short note is not to  address these low-regularity compatibility questions in full generality, as they are expected to become very subtle (and indeed have only recently been comprehensively answered even in the Riemannian setting, cf. \cite{MondinoRyborz}), 
but rather to investigate a peculiar regularity gap specifically in the equivalence of McCann's proof of the equivalence of Definitions \ref{def:classicNEC} and \ref{def:RobertNEC}, which is the main motivation behind his definition of the synthetic $\mathrm{NC}_q^e(N)$ condition: 
 While McCann formulates his full consistency result \cite[Thm.\ 31]{mccann2024synthetic} only for smooth metrics, a brief examination shows that $C^3$ would be sufficient for his proofs to go through without modifications as the central new result for the proof,  \cite[Thm.\ 26]{mccann2024synthetic}, which establishes the aforementioned equivalence of Definitions \ref{def:classicNEC} and \ref{def:RobertNEC}, is formulated for arbitrary $C^1$-tensor fields $F$ and thus applies to both $F=\mathrm{Ric}_g$ for a $C^3$-metric $g$ and to $F=\mathrm{Ric}^{N,V}_g$ for  the {\em $N$-Bakry-\'Emery-Ricci tensor} $$\mathrm{Ric}_g^{N,V}:=\mathrm{Ric}_g+ \mathrm{Hess}_g(V)+ \frac{1}{N-n} DV\otimes DV$$ with $g\in C^3$, $N\geq n$ and a weight $V\in C^3(M)$, which is assumed constant in case $n=N$ so that $\mathrm{Ric}_g^{n,V}=\mathrm{Ric}_g$. However, the given proof breaks down for continuous $F$, which means that the overall proof of equivalence already breaks down for $C^2$-metrics -- a regularity which is still very much considered classical in the field and for which the consensus is that results for smooth metrics should generally go through. As such it would be quite surprising if this degree of differentiability really were necessary. Moreover, the direction which becomes problematic is the step from Definition \ref{def:classicNEC} to Definition \ref{def:RobertNEC} which in the overall argument is relevant for going from the classical NEC to the $\mathrm{NC}_q^e(N)$ condition -- the direction which appears to be more easily shown to be robust under lowering regularity (cf. \cite{braun2024timelike}, but also the Riemannian case, where this direction has independently been established in \cite{KunzingerOhanyanVardabasso}  additionally to the most general equivalence of \cite{MondinoRyborz} mentioned above).\\

In this short note we show that McCann's proof can indeed be modified to show that Definitions \ref{def:classicNEC} and  \ref{def:RobertNEC} remain equivalent for $C^0$-tensor fields $F$ and thus can be applied to both $F=\mathrm{Ric}_g$ for a $C^2$-metric $g$ and $F=\mathrm{Ric}^{N,V}_g$ with $g\in C^2$, $N\geq n$ and a weight $V\in C^2(M)$ (cf. Remark \ref{ourrem28}). While this might not be the most general result possible, it confirms the common belief that -- maybe with some minor modifications to the proofs -- results established for smooth metrics can generally be recovered for $g\in C^2$.

\medskip

{\bf Notations and conventions } We want to emphasize that our signature convention for spacetimes is $(-,+,\dots,+)$ and as such unfortunately differs from the signature convention used in our main reference \cite{mccann2024synthetic}. 
Other definitions and conventions align with \cite{mccann2024synthetic}.

\section{Equivalence of Definitions \ref{def:classicNEC} and \ref{def:RobertNEC} for \texorpdfstring{$C^2$}{C2}-Lorentzian metrics}
\label{sec:nec-continuous-tensor-field-continuous-metric}

To establish the equivalence of Definitions \ref{def:classicNEC} and \ref{def:RobertNEC}, \cite{mccann2024synthetic} relies on the following slightly more general result, which \cite{mccann2024synthetic} formulates for $C^1$-tensor fields $F$ and smooth semi-Riemannian metrics and we are able to generalize to $C^0$-tensor fields $F$ and $C^0$ semi-Riemannian metrics $g$.

\begin{thm} \label{thm:null-bounds}
  Let $M$ be a smooth manifold with a $C^0$ semi-Riemannian metric $g$. Let \( F \) be a $(0,2)$-tensor field of regularity $C^{0}$ satisfying the null energy condition
  \[F(\v,\v) \geq 0  \quad \text{for all null vectors \( \v \)}.\]
  Then, for each compact subset \( Z \subseteq M \) there is a constant \( C_Z \in \mathbb{R} \) such that
  \begin{equation} \label{eq:null-bounds}
    F(\v,\v) \geq C_Z \cdot | g(\v,\v) | \tag{*}
  \end{equation}
  for all \( \v \in TZ:=TM|_Z \).
\end{thm}
Before going into the details of the proof, let us outline the general argument. The proof of Theorem \ref{thm:null-bounds} in \cite{mccann2024synthetic} essentially observes that the only way for the quotient $ \frac{F(\v,\v)}{|g(\v,\v)|} $ to diverge to $-\infty$ would be for $\v$ to become null (but non-zero) and then uses a Taylor expansion of the tensor fields \(F\) and \(g\) around the null bundle $L=\{\v\in TM:\,g(\v,\v)=0,\,\v\neq 0\}\subset TM$ to estimate that in reality the liminf of the quotient must always remain bounded from below even as one approaches the null bundle: Derivatives of $g(.,.):TM\to \R$ transverse to $L$ cannot vanish (which follows from non-degeneracy of $g$), so the denominator goes to zero linearly, and from the imposed non-negativity of $F$ on $L$ one can show that the negative part of the numerator must approach zero at least linearly as well. While it seems reasonable 
that the transverse direction used in the Taylor argument might be  picked in such a way that the base point $p$ remains fixed, 
in McCann's very general 
construction starting from an {\em arbitrary} Riemannian background metric this transverse derivative corresponds to a vector field $X\in \mathfrak{X}(TM)$ for which generally $d\pi \circ X\neq 0$ (where $\pi$ is the usual projection $\pi:TM \to M$). If this can be remedied, one may be able to use the fact that any tensor field is automatically continuously differentiable in vertical directions by multilinearity to recover the proof for $C^0$ tensor fields $F$. In principle the same trick also removes the need for differentiability of $g$ as long as one is sufficiently careful in the construction of a suitable vector field $X$ and in analyzing its flow, which both become slightly more challenging without differentiability of $g$. 

\begin{proof}[Proof of Theorem \ref{thm:null-bounds}.] For any compact $Z\subseteq M$ we set 
$$  C_Z 
      := \inf_{\v \in TM|_Z:\,g(\v,\v)\neq 0} \frac{F(\v, \v)}{g(\v, \v)}.$$
If $C_Z>-\infty$ this is clearly a suitable constant for \eqref{eq:null-bounds}. 

We first remark that if $g$ is Riemannian, there are no 
null vectors. Thus every $C^0$-tensor field satisfies the null energy condition and also \[ 
      C_Z 
      = \inf_{\v \in TM|_Z:\,\v\neq 0} \frac{F(\v, \v)}{g(\v, \v)} 
      = \inf_{\v \in  TM|_Z\,:\,g(\v,\v)=1} F(\v, \v) > -\infty.
    \]
    automatically. So let us assume that $g$ has at least one negative eigenvalue. 

   We next reduce the proof to the case where $M=\R^n$. Since $Z$ is compact, we can cover it by finitely many compact subsets $Z_i,\,i=1,\dots, N$, each of which is still contained in a coordinate neighbourhood $U_i\cong \R^n$. If we can show that for each compact $Z_i\subseteq U_i$ the constant \(C_{Z_i} \) defined as above is in \(\R\), then 
    \[
      C_Z = \min_{1 \leq i \leq N} C_{Z_i}\in \R
    \]
    is the desired constant for \(Z\). Thus we may, without loss of generality, assume that we have fixed global smooth coordinates in which $M=\R^n$ and that $Z$ is a compact subset of $\R^n$. This simplifies the setting and instead of choosing an arbitrary complete\footnote{Completeness is, however, not actually used in either McCann's or our proof.} 
    background Riemannian metric on $M$ as in \cite{mccann2024synthetic} we may simply use the Euclidean metric in the given coordinates as our background Riemannian metric, denoting it with $\langle ., .\rangle$ and the associated norm with $\|.\|$. To obtain a (then also complete) Riemannian metric on $TM$ \cite{mccann2024synthetic}  uses the Sasaki-Metric obtained from the arbitrarily chosen Riemannian background metric on $TM$. For us $TM=T\R^n=\R^n\times \R^n=\R^{2n}$ trivializes globally in its induced smooth coordinates, so we can simply equip it with the Euclidean metric in these coordinates, which we by a slight abuse of notation again denote by $\langle ., .\rangle$. 

As in \cite{mccann2024synthetic}, due to the 2-homogeneity of the bilinear forms \( F \) and \( g \), it suffices to 
consider the infimum over 
all \emph{unit} (with respect to the background Riemannian metric) vectors \( \mathbf{v}=(p,v) \in TM|_Z=Z \times \R^n \) with $g(\v,\v) \neq 0$ in the definition of $C_Z$. Since $M=\R^n$ and our background Riemannian metric is just the Euclidean one, the unit sphere bundle becomes $\R^n\times \S^{n-1}$, where $\S^{n-1}\subseteq \R^n$ is the standard round sphere. In particular this unit sphere bundle is globally trivial and a smooth $(2n-1)$-dimensional submanifold of $TM=\R^{2n}$. We set $S:=\R^n\times \S^{n-1}$ and decompose it into three disjoint subsets:
\[
      S = S_{-} \cup S_{0} \cup S_{+},
\]
corresponding to timelike, null, and spacelike vectors, respectively. Clearly, this decomposition depends on the $C^0$ semi-Riemannian metric $g$. We note that 
\[ 
S_0=\{(p,v)\in \R^n\times \S^{n-1}: g_p(v, v) = 0 \}
\]
is a $(2n-2)$-dimensional topological submanifold of $TM=\R^n \times \R^n$: Since $g$ is not differentiable we cannot argue via the tangent spaces to the null bundle and the unit-sphere bundle spanning $TM$ as in \cite{mccann2024synthetic}. Instead, we will show the desired property more directly. For any $p\in M=\R^n$, $g_p$ is a symmetric bilinear form on $\R^n$, hence can be orthogonally diagonalized, i.e. there exists an orthogonal matrix $U_p$ such that $U_p^T g_p U_p= \eta_{(k,n-k)} $, where $\eta_{(k,n-k)}$ is the standard semi-Riemannian inner product of signature $(k,n-k)$ on $\R^n$. Since $U_p$ consists of the eigenvectors of $g_p$, it depends continuously on $p$. 
This gives a {\em continuous} global change of coordinates $\psi: (p,v)\mapsto (p,U_pv)$ for $TM=\R^{2n}$ with continuous inverse $\psi^{-1}: (p,v')\mapsto (p, U_p^T v')$. For $(p,v)\in \R^n\times \R^n$ and $v':=U_pv$ we have $\|v\|=1 \iff \|v'\|=1$ since $U_p$ is orthogonal and $g_p(v,v)=0 \iff \eta_{(k,n-k)}(v',v')=0$ since $U_p$ diagonalizes $g_p$, 
so in the new coordinates $(p,v')$ the subset $S_0$ becomes just 
\begin{equation}\label{eq:psiS0}
    \psi(S_0)=\{(p,v')\in \R^n\times \R^n: \|v'\|=1 \text{ and } \eta_{(k,n-k)}(v',v')=0\},
\end{equation} which is known to be a smooth $(2n-2)$-dimensional submanifold of $\R^{n}\times \R^n$. Hence $S_0$ 
is a $(2n-2)$-dimensional topological submanifold of $TM$.\\

After this set-up, let us continue with the proof. As in \cite{mccann2024synthetic} we assume for the sake of contradiction that the theorem's assertion is false, i.e., that $C_Z = -\infty$.
Then 
there exists a sequence \( (p_{\seqIndex},v_{\seqIndex})_{\seqIndex \in\mathbb{N}} \) in $S|_Z:=Z\times \S^{n-1}\subset S$ with $g_{p_{\seqIndex}}(v_{\seqIndex},v_{\seqIndex}) \neq 0$ such that
\begin{equation} \label{eq:contradiction-assumption}
  r(p_{\seqIndex},v_{\seqIndex}) := \frac{F_{p_{\seqIndex}}(v_{\seqIndex},v_{\seqIndex})}{\left|g_{p_{\seqIndex}}(v_{\seqIndex}, v_{\seqIndex})\right|} \to -\infty.
\end{equation}
By compactness of $S|_Z$ a subsequence of \( (p_{\seqIndex},v_{\seqIndex}) \), which we again denote by \( (p_{\seqIndex},v_{\seqIndex}) \),  must converge to some limit \( (p_{\infty},v_{\infty}) \in S|_Z \). Since \( r(p_{\seqIndex},v_{\seqIndex}) \) diverges to \( -\infty \) and $F$ is bounded on $S|_Z$, this forces \( g_{p_{\seqIndex}}(v_{\seqIndex}, v_{\seqIndex}) \to 0 \), implying that \( (p_{\infty},v_{\infty})\in S_0|_Z \).\\

Our goal is to essentially think of the $\R^n$-component $p$ as a {\em parameter } instead of a coordinate as we approach $(p_{\infty},v_{\infty})$. This will be helpful because these directions are the only source for the non-smoothness: For any {\em fixed} $p\in M=\R^n$, $S_0|_p$ is a smooth $(n-2)$-dimensional submanifold of both $S|_p=\{p\}\times \S^{n-1}$ and  $T_pM=\{p\}\times \R^n$ and both $g_p$ and $F_p$ are smooth on $T_pM\times T_pM$ by bilinearity. Instead of the normal/Fermi coordinates around $S_0$ (based on the Sasaki metric on $TM$ corresponding to the arbitrarily chosen Riemannian background metric on $M$) that \cite{mccann2024synthetic} uses to get a handle on the distance to $S_0$ (which will potentially mix vertical and horizontal components), we will use the flow of a {\em continuous vertical vector field} $X\in \mathfrak{X}(TM)$, i.e. $X=0+\sum_{i=1}^n X^i \frac{\partial}{\partial v^i}$ for continuous $X^i:TM\to \R$, 
such that $X|_{S}$  
 remains tangent to the unit sphere bundle $S$ but such that $X_{(p,v)}$ is non-tangent to the (smooth) fiber $L_p=\{(p,v)\in T_pM=\{p\}\times \R^n: g_p(v,v)=0\}$ of the null cone  of the semi-Riemannian metric $g$ for any $(p,v)\in S_0$.

We can give a suitable $X$ very explicitly by making the following ansatz: We start with $$Y_{(p,v)}:=\displaystyle\sum_{i=1}^{n} \dfrac{\partial g}{\partial v^i} \cdot \left.\dfrac{\partial}{\partial v^i}\right|_{(p,v)}=2 \displaystyle\sum_{i=1}^{n}\left(\sum_{j=1}^{n} g_{ij}(p) v^j\right) \left.\dfrac{\partial}{\partial v^i}\right|_{(p,v)},$$ which is equal to the vertical projection of $$\gradRiem g_{(p,v)} = \sum_{i=1}^{n}\left(\sum_{j,k=1}^{n} \frac{\partial g_{jk}(p)}{\partial p^i} \cdot v^jv^k \right)\left.\frac{\partial}{\partial p^i}\right|_{(p,v)}
    + \sum_{i=1}^{n}\left(2\sum_{j=1}^{n} g_{ij}(p) v^j \right) \left.\frac{\partial}{\partial v^i}\right|_{(p,v)}$$  if $g \in C^1$, but is also perfectly well-behaved for continuous $g$. We then project onto $TS$ to obtain
\begin{align} \label{eq:vector-field}
  X_{(p,v)} &:= Y_{(p,v)} - \langle Y_{(p,v)} ,n_{(p,v)} \rangle  \cdot n_{(p,v)} \notag\\
  &= 2 \displaystyle\sum_{i=1}^{n}\left(\sum_{j=1}^{n} g_{ij}(p) v^j - g_p(v,v) v^i\right) \left.\dfrac{\partial}{\partial v^i}\right|_{(p,v)},
\end{align}
where $n_{(p,v)}:= v^i \dfrac{\partial}{\partial v^i} $ is the unit normal to $S$ at $(p,v)$ for $(p,v)\in S\subset TM$. This vector field is clearly continuous and vertical on $TM$ and $X|_{S}$ is tangent to $S$ by construction. It remains to check
\textit{non-tangency to $L_p$} for any $(p,v)\in S_0$. Since $T_vL_p = \{ w \in T_pM \mid g_p(v, w) = 0 \}$ it suffices to show that $g_p(v, X_{(p,v)}) \neq 0$. We compute 
$$ g_p(v, X_{(p,v)}) = 2 \sum_{i,j, k} g_{ij}(p) g_{ki}(p)  v^j v^k  = 2 \, \langle g_p\cdot v, g_p\cdot v\rangle \neq 0,$$
where $g_p\cdot v:=(\sum_j g_{ij}(p) v^j)_{i=1,\dots,n}$ and we used non-degeneracy of $g$ and that $v\neq 0$.\\

We now wish to define a ``projection'' of 
points $(p,v)$ in $S$ near $S_0$ onto $S_0$ by using the integral curves of $X$. 
Since $X$ is overall merely continuous we cannot directly use standard results such as flowout coordinates. Instead we will solve the ODE for the integral curves directly: 

Because $X$ is vertical, any integral curve $\gamma: I\to S=\R^n\times \S^{n-1}$ of $X$ starting in some $(p_0,v_0)\in S_0\subset \R^n\times \S^{n-1}$ must remain in the same fiber, i.e. $\gamma(t)=(p_0,v(t))$ for all $t$, and $v:I\to \S^{n-1}\subset \R^n$ must solve the fiberwise ODE:
\begin{equation} \label{eq:nec-3-integral-curves}
v: I \to S\, ,\ 
  \left\{\begin{array}{ll}
v(0) = v_0, \\
v'(t) = 2\, g_{p_0}\cdot v(t) - 2\,g_{p_0}(v(t),v(t)) \cdot v(t).
\end{array}\right.
\end{equation}

For fixed $p_0$, the right-hand side is smooth in $v$, so by Picard-Lindelöf, for each $(p_0, v_0) \in S_0$, there exists a unique solution $\gamma_{p_0}(t) = (p_0, v(t))$ defined on a maximal interval $I_{p_0,v_0} \ni 0$. We notice that equation \eqref{eq:nec-3-integral-curves} is almost linear except for the projection term to stay on the sphere. Since the linear equation is solved by $t\mapsto e^{2t\, g_{p_0}}\cdot v_0$ for all $t\in \R$, where $g_{p_0}$ is the matrix $(g_{ij}(p_0))_{i,j=1,\dots,n}$ and $e^{2tg}$ is the matrix exponential of $2tg$, we make the following ansatz for the solution of (\ref{eq:nec-3-integral-curves}) and the associated flow map
\begin{equation} \label{eq:flow}
  \Phi: \R \times S_0 \to S, \qquad \Phi(t, (p_0, v_0)) = (p_0,v(t)) = \begin{pmatrix} p_0,\, \dfrac{e^{2t \cdot g_{p_0}} v_0}{\|e^{2t \cdot g_{p_0}} v_0\|} \end{pmatrix},
\end{equation}
where $\|\cdot\|$ is the Euclidean norm. It is straightforward to check that this indeed solves \eqref{eq:nec-3-integral-curves} (see Appendix \ref{secA1}, Lemma \ref{thm:flow-map-solves-ode}). Clearly $\Phi$ is globally defined and continuous on $\mathbb{R} \times S_0$. Furthermore, one can explicitly show that $\Phi$ is injective (see Appendix \ref{secA1}, Lemma \ref{thm:flow-map-is-injective}).

\medskip

Since \(\mathbb{R} \times S_0\) and \(S\) are both \((2n-1)\)-dimensional topological manifolds and \(\Phi : \mathbb{R} \times S_0\to S \) is continuous and injective, invariance of domain implies that \(V:=\Phi(\R\times S_0) \subseteq S \) is open in $S$ and $\Phi: \R \times S_0 \to V$ is a \emph{homeomorphism}. Since $S_0\subset V$, the points \((p_{\seqIndex}, v_{\seqIndex})\) of the sequence converging to $(p_\infty,v_\infty)\in S_0$ along which $r$ diverges to $-\infty$ lie in $V$ for all large enough $m$. Define $$
(t_{\seqIndex},(\hat{p}_{\seqIndex},\hat{v}_{\seqIndex})):= \Phi^{-1}((p_{\seqIndex}, v_{\seqIndex})) \in \R \times S_0.$$
Since $p$ is constant along the integral curves, we immediately get $\hat{p}_m=p_m$ for all $m$. Further, $(t_\infty, (\hat{p}_\infty,\hat{v}_\infty))=(0,(p_\infty,v_\infty))$ since $(p_\infty,v_\infty)\in S_0$.

Continuity of \(\Phi^{-1}\) implies \(t_{\seqIndex} \to 0\) and \((p_{\seqIndex}, \hat{v}_{\seqIndex}) \to (p_\infty, v_\infty)\) as \(m \to \infty\).

\begin{figure}[H] \centering
  
\tikzset {_yhrqgvb0m/.code = {\pgfsetadditionalshadetransform{ \pgftransformshift{\pgfpoint{89.1 bp } { -128.7 bp }  }  \pgftransformscale{1.32 }  }}}
\pgfdeclareradialshading{_t82q9yf8r}{\pgfpoint{-72bp}{104bp}}{rgb(0bp)=(1,1,1);
rgb(0bp)=(1,1,1);
rgb(25bp)=(0.74,0.73,0.73);
rgb(400bp)=(0.74,0.73,0.73)}
\tikzset{_ngv7cgngq/.code = {\pgfsetadditionalshadetransform{\pgftransformshift{\pgfpoint{89.1 bp } { -128.7 bp }  }  \pgftransformscale{1.32 } }}}
\pgfdeclareradialshading{_jimnshe3y} { \pgfpoint{-72bp} {104bp}} {color(0bp)=(transparent!0);
color(0bp)=(transparent!0);
color(25bp)=(transparent!90);
color(400bp)=(transparent!90)} 
\pgfdeclarefading{_six8e7ykn}{\tikz \fill[shading=_jimnshe3y,_ngv7cgngq] (0,0) rectangle (50bp,50bp); } 
\tikzset{every picture/.style={line width=0.75pt}} 

\begin{tikzpicture}[x=0.75pt,y=0.75pt,yscale=-1,xscale=1]
  \begin{scope}
    \clip (50,0) rectangle (300,150);


\path  [shading=_t82q9yf8r,_yhrqgvb0m,path fading= _six8e7ykn ,fading transform={xshift=2}] (42.68,149.7) .. controls (42.68,81.6) and (97.89,26.4) .. (165.98,26.4) .. controls (234.08,26.4) and (289.29,81.6) .. (289.29,149.7) .. controls (289.29,217.8) and (234.08,273) .. (165.98,273) .. controls (97.89,273) and (42.68,217.8) .. (42.68,149.7) -- cycle ; 
 \draw   (42.68,149.7) .. controls (42.68,81.6) and (97.89,26.4) .. (165.98,26.4) .. controls (234.08,26.4) and (289.29,81.6) .. (289.29,149.7) .. controls (289.29,217.8) and (234.08,273) .. (165.98,273) .. controls (97.89,273) and (42.68,217.8) .. (42.68,149.7) -- cycle ; 

\draw    (159,42.28) .. controls (204.43,40.29) and (236.33,66) .. (227.6,76.28) ;
\draw [fill={rgb, 255:red, 228; green, 41; blue, 41 }  ,fill opacity=1 ]   (50.29,41) -- (273.41,255.2) ;
\draw    (275.29,44.6) -- (61.54,257) ;
\draw [line width=1.5]    (77.88,65.37) .. controls (110.37,89.67) and (230.37,86.07) .. (252.87,65.37) ;
\draw [line width=1.5]    (79.75,236.37) .. controls (112.25,260.67) and (232.25,257.07) .. (254.75,236.37) ;
\draw  [dash pattern={on 4.5pt off 4.5pt}]  (79.75,236.37) .. controls (110.37,212.07) and (222.87,210.27) .. (254.75,236.37) ;
\draw  [fill={rgb, 255:red, 0; green, 0; blue, 0 }  ,fill opacity=1 ] (220,63.28) .. controls (220,61.84) and (221.16,60.68) .. (222.6,60.68) .. controls (224.04,60.68) and (225.2,61.84) .. (225.2,63.28) .. controls (225.2,64.72) and (224.04,65.88) .. (222.6,65.88) .. controls (221.16,65.88) and (220,64.72) .. (220,63.28) -- cycle ;
\draw  [fill={rgb, 255:red, 0; green, 0; blue, 0 }  ,fill opacity=1 ] (224,69.28) .. controls (224,67.84) and (225.16,66.68) .. (226.6,66.68) .. controls (228.04,66.68) and (229.2,67.84) .. (229.2,69.28) .. controls (229.2,70.72) and (228.04,71.88) .. (226.6,71.88) .. controls (225.16,71.88) and (224,70.72) .. (224,69.28) -- cycle ;
\draw [color={rgb, 255:red, 65; green, 117; blue, 5 }  ,draw opacity=1 ]   (161.6,42.28) .. controls (170.43,56.29) and (167.43,75.29) .. (164.6,82.28) ;
\draw [color={rgb, 255:red, 65; green, 117; blue, 5 }  ,draw opacity=1 ]   (182.6,43.28) .. controls (189.83,51.29) and (191.2,71.6) .. (189.6,81.28) ;
\draw [color={rgb, 255:red, 65; green, 117; blue, 5 }  ,draw opacity=1 ]   (201.6,49.68) .. controls (210.2,55) and (209.2,75.6) .. (207.73,80) ;
\draw  [fill={rgb, 255:red, 74; green, 144; blue, 226 }  ,fill opacity=1 ] (162,82.28) .. controls (162,80.84) and (163.16,79.68) .. (164.6,79.68) .. controls (166.04,79.68) and (167.2,80.84) .. (167.2,82.28) .. controls (167.2,83.72) and (166.04,84.88) .. (164.6,84.88) .. controls (163.16,84.88) and (162,83.72) .. (162,82.28) -- cycle ;
\draw  [fill={rgb, 255:red, 74; green, 144; blue, 226 }  ,fill opacity=1 ] (187,81.28) .. controls (187,79.84) and (188.16,78.68) .. (189.6,78.68) .. controls (191.04,78.68) and (192.2,79.84) .. (192.2,81.28) .. controls (192.2,82.72) and (191.04,83.88) .. (189.6,83.88) .. controls (188.16,83.88) and (187,82.72) .. (187,81.28) -- cycle ;
\draw  [fill={rgb, 255:red, 74; green, 144; blue, 226 }  ,fill opacity=1 ] (205.13,80) .. controls (205.13,78.56) and (206.3,77.4) .. (207.73,77.4) .. controls (209.17,77.4) and (210.33,78.56) .. (210.33,80) .. controls (210.33,81.44) and (209.17,82.6) .. (207.73,82.6) .. controls (206.3,82.6) and (205.13,81.44) .. (205.13,80) -- cycle ;
\draw  [fill={rgb, 255:red, 74; green, 144; blue, 226 }  ,fill opacity=1 ] (216,78.28) .. controls (216,76.84) and (217.16,75.68) .. (218.6,75.68) .. controls (220.04,75.68) and (221.2,76.84) .. (221.2,78.28) .. controls (221.2,79.72) and (220.04,80.88) .. (218.6,80.88) .. controls (217.16,80.88) and (216,79.72) .. (216,78.28) -- cycle ;
\draw  [fill={rgb, 255:red, 74; green, 144; blue, 226 }  ,fill opacity=1 ] (219,78.28) .. controls (219,76.84) and (220.16,75.68) .. (221.6,75.68) .. controls (223.04,75.68) and (224.2,76.84) .. (224.2,78.28) .. controls (224.2,79.72) and (223.04,80.88) .. (221.6,80.88) .. controls (220.16,80.88) and (219,79.72) .. (219,78.28) -- cycle ;
\draw  [fill={rgb, 255:red, 74; green, 144; blue, 226 }  ,fill opacity=1 ] (223,77.28) .. controls (223,75.84) and (224.16,74.68) .. (225.6,74.68) .. controls (227.04,74.68) and (228.2,75.84) .. (228.2,77.28) .. controls (228.2,78.72) and (227.04,79.88) .. (225.6,79.88) .. controls (224.16,79.88) and (223,78.72) .. (223,77.28) -- cycle ;
\draw  [fill={rgb, 255:red, 208; green, 2; blue, 27 }  ,fill opacity=1 ] (225,76.28) .. controls (225,74.84) and (226.16,73.68) .. (227.6,73.68) .. controls (229.04,73.68) and (230.2,74.84) .. (230.2,76.28) .. controls (230.2,77.72) and (229.04,78.88) .. (227.6,78.88) .. controls (226.16,78.88) and (225,77.72) .. (225,76.28) -- cycle ;
\draw [color={rgb, 255:red, 65; green, 117; blue, 5 }  ,draw opacity=1 ]   (213.6,54.28) .. controls (218.2,59.6) and (221.2,69.6) .. (219,78.28) ;
\draw  [fill={rgb, 255:red, 0; green, 0; blue, 0 }  ,fill opacity=1 ] (159,42.28) .. controls (159,40.84) and (160.16,39.68) .. (161.6,39.68) .. controls (163.04,39.68) and (164.2,40.84) .. (164.2,42.28) .. controls (164.2,43.72) and (163.04,44.88) .. (161.6,44.88) .. controls (160.16,44.88) and (159,43.72) .. (159,42.28) -- cycle ;
\draw  [fill={rgb, 255:red, 0; green, 0; blue, 0 }  ,fill opacity=1 ] (180,43.28) .. controls (180,41.84) and (181.16,40.68) .. (182.6,40.68) .. controls (184.04,40.68) and (185.2,41.84) .. (185.2,43.28) .. controls (185.2,44.72) and (184.04,45.88) .. (182.6,45.88) .. controls (181.16,45.88) and (180,44.72) .. (180,43.28) -- cycle ;
\draw  [fill={rgb, 255:red, 0; green, 0; blue, 0 }  ,fill opacity=1 ] (199,49.28) .. controls (199,47.84) and (200.16,46.68) .. (201.6,46.68) .. controls (203.04,46.68) and (204.2,47.84) .. (204.2,49.28) .. controls (204.2,50.72) and (203.04,51.88) .. (201.6,51.88) .. controls (200.16,51.88) and (199,50.72) .. (199,49.28) -- cycle ;
\draw  [fill={rgb, 255:red, 0; green, 0; blue, 0 }  ,fill opacity=1 ] (211,54.28) .. controls (211,52.84) and (212.16,51.68) .. (213.6,51.68) .. controls (215.04,51.68) and (216.2,52.84) .. (216.2,54.28) .. controls (216.2,55.72) and (215.04,56.88) .. (213.6,56.88) .. controls (212.16,56.88) and (211,55.72) .. (211,54.28) -- cycle ;

\draw (227.6,80.28) node [anchor=north west][inner sep=0.75pt]  [font=\scriptsize,color={rgb, 255:red, 208; green, 2; blue, 27 }  ,opacity=1 ] [align=left] {$\displaystyle ( p_{\infty } ,v_{\infty })$};
\draw (115,37) node [anchor=north west][inner sep=0.75pt]  [font=\scriptsize] [align=left] {$\displaystyle ( p_{m} ,v_{m})$};
\draw (122,87.28) node [anchor=north west][inner sep=0.75pt]  [font=\scriptsize,color={rgb, 255:red, 74; green, 144; blue, 226 }  ,opacity=1 ] [align=left] {$\displaystyle \left(p_{\seqIndex} ,\hat{v}_{m}\right)$};
\draw (150,54) node [anchor=north west][inner sep=0.75pt]  [font=\footnotesize,color={rgb, 255:red, 65; green, 117; blue, 5 }  ,opacity=1 ] [align=left] {$\displaystyle \gamma _{m}$};
\end{scope}
\end{tikzpicture}
  \caption{Approaching $\left(p_{\infty},v_{\infty}\right)$ along the null projections $(p_{\seqIndex},\hat{v}_{\seqIndex})$ at $S_0$ (for illustrative purposes $p_{\seqIndex} = p_{\infty}$).}
  \label{fig:unit-sphere-bundle-taylor-approximation-flowout}
\end{figure}
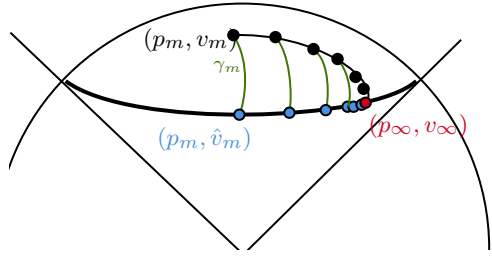

Denote the integral curve starting in $(p_m,\hat{v}_m)$ by $\gamma_m$, i.e., \begin{equation}\label{eq:vhateq}
   \gamma_{\seqIndex}(t) := \Phi(t,(p_m,\hat{v}_m))=(p_m, \hat{v}_m(t)) \text{ with } \hat{v}_m(t):=\dfrac{e^{2t \cdot g_{p_m}} \hat{v}_m}{\|e^{2t \cdot g_{p_m}} \hat{v}_m\|} 
\end{equation}
as in \eqref{eq:flow}. Then both $F\circ \gamma_m:=F_{p_m}(\hat{v}_m(.),\hat{v}_m(.)) :\R \to \R$ and $g\circ \gamma_m:=g_{p_m}(\hat{v}_m(.),\hat{v}_m(.)) :\R \to \R$ are {\em smooth} and we can Taylor expand
\begin{align}\frac{F(\gamma_{\seqIndex}(t_{\seqIndex}))}{\left|g(\gamma_{\seqIndex}(t_{\seqIndex}))\right|} &=  \frac{(F \circ \gamma_{\seqIndex})(0) + t_{\seqIndex} \cdot  (F \circ \gamma_{\seqIndex})^{\prime}(0) + \frac{t_{\seqIndex}^2}{2} (F \circ \gamma_{\seqIndex})^{\prime\prime}(\xi_{\seqIndex})}
   {\left|\, (g \circ \gamma_{\seqIndex})(0) + t_{\seqIndex} \cdot  (g \circ \gamma_{\seqIndex})^{\prime}(0) + \frac{t_{\seqIndex}^2}{2!} (g \circ \gamma_{\seqIndex})^{\prime\prime}(\nu_{\seqIndex})\,\right|} \nonumber \\                                
 & \geq \frac{t_{\seqIndex} \cdot (F \circ \gamma_{\seqIndex})^{\prime}(0) + \frac{t_{\seqIndex}^2}{2} (F \circ \gamma_{\seqIndex})^{\prime\prime}(\xi_{\seqIndex})}
 {\left|  t_{\seqIndex} \cdot  (g \circ \gamma_{\seqIndex})^{\prime}(0) + \frac{t_{\seqIndex}^2}{2} (g \circ \gamma_{\seqIndex})^{\prime\prime}(\nu_{\seqIndex})\,\right|} \nonumber  \\
 & \geq - \Bigg| \frac{(F \circ \gamma_{\seqIndex})^{\prime}(0) + \frac{t_{\seqIndex}}{2} (F \circ \gamma_{\seqIndex})^{\prime\prime}(\xi_{\seqIndex})}
 { (g \circ \gamma_{\seqIndex})^{\prime}(0) + \frac{t_{\seqIndex}}{2} (g \circ \gamma_{\seqIndex})^{\prime\prime}(\nu_{\seqIndex})} \Bigg|
\end{align}
for suitable $\xi_{\seqIndex},\nu_{\seqIndex} \in [0,t_{\seqIndex})$, where we used $(p_\infty,v_\infty)\in S_0$ and $F$ satisfying the null energy condition in the first inequality.

Since $\hat{v}_m'(t)$ and $ \hat{v}_m''(t)$ are continuous functions of $t$ and $(p_m,v_m)$ (which follows either from the ODE \eqref{eq:nec-3-integral-curves} or the explicit form \eqref{eq:vhateq}) and $F$ is continuous, also $(F\circ \gamma_m)'(t)=2 F_{p_m}(\hat{v}_m(t),\hat{v}_m'(t))$ and $(F\circ \gamma_m)''(t)=2 F_{p_m}(\hat{v}'_m(t),\hat{v}_m'(t))+2 F_{p_m}(\hat{v}_m(t),\hat{v}_m''(t))$ are jointly continuous in $t$ {\em and} $(p_m,v_m)$. The same holds for $(g\circ \gamma_m)'(t)$ and $(g\circ \gamma_m)''(t)$. Since $t_m$ converges to zero as $m \to \infty$, also $\nu_m\to 0$  we obtain
\begin{align}
    \lim_{m\to \infty} (g \circ \gamma_{\seqIndex})^{\prime}(0) + \frac{t_{\seqIndex}}{2} (g \circ \gamma_{\seqIndex})^{\prime\prime}(\nu_{\seqIndex}) &= (g\circ \gamma_\infty)'(0)= 2 g_{p_\infty}(v_\infty, \gamma_\infty'(0)) \nonumber \\
    & =4 g_{p_\infty}(v_\infty, g_{p_\infty}\cdot v_\infty)=4 \langle g_{p_\infty}\cdot v_\infty, g_{p_\infty}\cdot v_\infty \rangle
\end{align} 
since  $\gamma_\infty'(0)= 2\,g_{p_\infty}\cdot v_\infty$ by the ODE \eqref{eq:nec-3-integral-curves} or the explicit form \eqref{eq:vhateq}. This is always non-zero since $v_\infty \in S_0$, so in particular $v_\infty\neq 0$, and $g$ is non-degenerate. Because the denominator does not converge to zero, we obtain
$$\lim_{m\to \infty} r(p_m,v_m)= \lim_{\seqIndex \to \infty} \frac{F(\gamma_{\seqIndex}(t_{\seqIndex}))}{\left|g(\gamma_{\seqIndex}(t_{\seqIndex}))\right|}\geq - \Bigg|  \frac{(F\circ\gamma_\infty)'(0)}{4 \langle g_{p_\infty}\cdot v_\infty, g_{p_\infty}\cdot v_\infty \rangle} \Bigg| \in \R$$
for the ratio. This contradicts the assumption that $ r(p_m,v_m)\to -\infty$ as $m\to \infty$ and finishes the proof. 
\end{proof}

Using Theorem \ref{thm:null-bounds}, we also recover \cite[Cor.\ 27]{mccann2024synthetic} and \cite[Rem.\ 28]{mccann2024synthetic} in this generality with exactly the same proof as in \cite{mccann2024synthetic}.

\begin{cor}\label{cor:Necvstl}
     Let $M$ be a smooth manifold with a $C^0$ semi-Riemannian metric $g$ and let \( F \) be a $C^{0}$-tensor field.  Then 
  \[F(\v,\v) \geq 0  \quad \text{for all null vectors \( \v \)}.\]
  if and only if for each compact subset \( Z \subseteq M \) there is a constant \( C^T_Z=-C_Z \in \mathbb{R} \) such that
  \begin{equation} \label{eq:vartlbound}
    F(\v,\v) \geq C^T_Z \cdot g(\v,\v)  
  \end{equation}
  for all {\em timelike} \( \v \in TZ \).
\end{cor}

\begin{rem}[Weighted null energy vs weighted Ricci bounds]\label{ourrem28} Applying the previous corollary either to $F=\mathrm{Ric}_g$ for a $C^2$ Lorentzian metric or to 
$$\mathrm{Ric}_g^{N,V}:=\mathrm{Ric}_g+ \mathrm{Hess}_g(V)+ \frac{1}{N-n} DV\otimes DV$$ for $g\in C^2$, $N\geq n$ and a weight $V\in C^2(M)$ (which is assumed constant in case $n=N$),
shows that non-negativity of the (weighted) Ricci tensor in null directions is
equivalent to a local lower bound on the (weighted) Ricci curvature in
timelike directions.
\end{rem}

\section{Discussion of implications for the equivalence to the synthetic NEC of \cite{mccann2024synthetic}}

Having established the equivalence of Definitions \ref{def:classicNEC} and \ref{def:RobertNEC} in the desired regularities, we can now turn to discuss the implications for the equivalence of the classical null energy condition, cf. Definition \ref{def:classicNEC}, to 
McCann's synthetic null energy condition defined below.
 
\begin{defi}[Definition 29 in \cite{mccann2024synthetic}]
Given $N>0$ and $0<q<1$, a proper measured causally geodesic space $(M, d, \ell, m)$
is said to satisfy a (weak) null energy-dimension condition, $(w)\mathrm{NC}^e_q (N)$, (with dimension $N>0$ and $0<q<1$) if 
for any compact $Z\subseteq M$ there exists a constant $C_Z\in \R$ such that the metric measure spacetime $J(Z,Z):=J^+(Z)\cap J^-(Z)$ satisfies the respective timelike curvature dimension condition $(w)\mathrm{TCD}^e_q(C_Z,N)$ with curvature bound $C_Z$ and dimension $N>0$.
\end{defi}

In order to not get bogged down with having to consider various technicalities and cases, we restrict ourselves to considering $N\in [n,\infty)$ below, but similar considerations as in \cite[Thm.\ 25]{mccann2024synthetic} should give that, like the corresponding \cite[Thm.\ 31]{mccann2024synthetic}, the result remains true for any  $N\in (0,\infty]$ if $N\ge n$ is added to $(iii)$.

\begin{thm}[Consistency with the null energy condition]\label{T:NECconsistency}
Let $(M^n,g)$ be a globally hyperbolic spacetime with a $C^2$-Lorentzian metric $g$, $0<q<1$, and let $\ell$ and $\operatorname{vol}_g$ denote the induced Lorentzian time separation function and volume measure.
Fix $N \in (0,\infty]$ and $dm = e^{-V}d\operatorname{vol}_g$ for some $V \in C^2(M)$, and a complete auxiliary Riemannian metric $\tilde g$ on $M^n$ which induces a distance $d$.  
Consider
\begin{enumerate}[(i)]
\item $(M,d,\ell,m) \in \operatorname{wNC}^e_q(N)$
\item $(M,d,\ell,m) \in \operatorname{NC}^e_q(N)$ and
\item every null vector $(v,z) \in TM$ satisfies
$$
\operatorname{Ric}^{(N,V)}(v,v) \ge 0.
$$
\end{enumerate}
Then $(iii)\implies (ii) \implies (i)$.
\end{thm}

\begin{proof}
    We follow the argument in the proof of the corresponding result \cite[Thm.\ 31]{mccann2024synthetic}. That $(ii)\implies (i)$ holds for general metric measure spacetimes is immediate from the definitions.

    For the proof of $(iii)\implies (ii) $, we first note that even for globally hyperbolic $C^{1,1}$-metrics $g$ (and $C^1$-weights $V$) timelike Ricci curvature bounds imply the $\mathrm{TCD}_q^e(K,N)$ condition: 
   In this case \cite[Thm.\ 3.2]{braun2024timelike} shows that $\mathrm{TCD}_q(K,N)$ is satisfied and \cite[Prop.\ 3.6]{Braunequivalencies} 
   shows that $\mathrm{TCD}_q(K,N)$ implies $\mathrm{TCD}_q^*(K,N)$ (noting that coming from a globally hyperbolic spacetime with a $C^{1,1}$-metric our metric measure spacetime indeed satisfies \cite[Assumption 3.1]{Braunequivalencies}). 
   Lastly by \cite[Thm.\ 1.5]{Braunequivalencies} the latter is equivalent to $\mathrm{TCD}_q^e(K,N)$ for timelike non-branching 
  spaces (which is satisfied for $C^{1,1}$-metrics because this regularity still guarantees uniqueness of geodesics by classical ODE theory) satisfying \cite[Assumption 3.1]{Braunequivalencies}. Combining this with Remark \ref{ourrem28} essentially proves $(iii)\implies (ii)$. The only remaining minor issue is that McCann's definition works with the closed globally hyperbolic metric measure spacetimes $J(Z,Z)$ instead of the open $I(Z,Z)$ ones and only the latter really are Lorentzian manifolds in the classical sense. This could either be argued away as in \cite{mccann2024synthetic}, or -- if we want to stick with applying the low-regularity versions \cite{braun2024timelike, Braunequivalencies} directly -- by noting that $J(Z,Z)\subset I(\tilde{Z},\tilde{Z})$ 
  for any compact $\tilde{Z}$ with $Z$ contained in the interior of  $\tilde{Z}$.
  Since the spaces considered are globally hyperbolic and the diamonds are causally convex, this implies that $J(Z,Z)$ satisfies the $\mathrm{TCD}_q^e(K,N)$ condition with the same constant $K=C_{J(\tilde{Z},\tilde{Z})}$ as $I(\tilde{Z},\tilde{Z})$ (obtained from applying Remark \ref{ourrem28} to the compact set $J(\tilde{Z},\tilde{Z})$).
\end{proof}

\begin{rem} The remaining implication $(i)\implies (iii)$  of the full equivalence does not require our improved version of Theorem \ref{thm:null-bounds} and purely rests on the folklore of the standard arguments for establishing the equivalence of the various $\mathrm{TCD}$ conditions to timelike Ricci curvature bounds still going through for $C^2$-metrics (and $C^2$-weights). 
We were careful to not rely on any arguments of this form in the proof of Theorem \ref{T:NECconsistency} above, however, given that \cite[Thm. 8.5]{mccanntcd} (or alternatively one of the results \cite{braun2024optimalfinsler} or \cite{mondino2022optimal-transport-einstein}, 
where this might be easier to see) 
indeed goes through for $C^2$-metrics,  
$(i)\implies (iii)$ follows exactly as in \cite[Thm.\ 31]{mccann2024synthetic}. 
\end{rem}

\backmatter

\bmhead{Acknowledgements}
This article originated from work on YG’ Masters thesis at the University of Hamburg. We would like to thank Mathias Braun for helpful comments. 
MG acknowledges acknowledges support by the Deutsche Forschungsgemeinschaft (DFG, German Research Foundation) under Germany’s Excellence Strategy -- EXC 2121 ''Quantum Universe'' -- 390833306.






\begin{appendices}

\section{Properties of the flow map \texorpdfstring{$\Phi$}{Phi} \texorpdfstring{from \eqref{eq:flow}}{}} \label{secA1}

\begin{lem} \label{thm:flow-map-solves-ode}
  Let $S:= \mathbb{R}^n \times \mathbb{S}^{n-1}$, $S_0 := \{(p,\hat{v}) \in S \mid g(\hat{v},\hat{v})=0\}$ and set
    \begin{equation*}
        \Phi: \mathbb{R}\times S_0 \to S, \qquad \Phi(t, (p, \hat{v}))  
        := \begin{pmatrix} p,\, \dfrac{e^{2t \cdot g_p} \hat{v}}{\|e^{2t \cdot g_p} \hat{v}\|} \end{pmatrix}\equiv (p,v_{(p, \hat{v})}(t)).
    \end{equation*}
  Then for any initial data $(p,\hat{v})\in S_0$ the curve $t\mapsto v_{(p,\hat{v})}$ solves the ODE 
    \begin{equation*}
        v: I \to S\, ,\ \quad\quad          
        v'(t) = 2\, g_{p_0}\cdot v(t) - 2\,g_{p_0}(v(t),v(t)) \cdot v(t).   
    \end{equation*}
    from \eqref{eq:nec-3-integral-curves} with initial data $\Phi(0,(p,\hat{v}))=(p,\hat{v})$.
\end{lem}

\begin{proof}
The proof is a straightforward computation, but we include detailed steps for convenience. Setting $w(t):=e^{2t g_p} \hat{v}$, we obtain $\dot{w}(t)=2g_p\cdot w(t)$ and 
\begin{align*}
    \frac{d}{dt} \frac{1}{\|w(t)\|} 
    &=\frac{d}{dt} \frac{1}{\sqrt{w(t)^T w(t)}}
    = -\frac{1}{2}\frac{1}{\sqrt{w(t)^T w(t)}^{3}} \cdot 2\,w(t)^T \dot{w}(t)\\
    &= -\frac{2}{\|w(t)\|}\, g_p\left(\frac{w(t)}{\|w(t)\|}, \frac{w(t)}{\|w(t)\|}\right).
\end{align*}
So
\begin{align*}
\frac{d}{dt} \frac{w(t)}{\|w(t)\|}
=2g_p\cdot \frac{w(t)}{\|w(t)\|} - 2\,\frac{w(t)}{\|w(t)\|}\, g_p\left(\frac{w(t)}{\|w(t)\|}, \frac{w(t)}{\|w(t)\|}\right)
\end{align*}
Since $v(t)=w(t) / \|w(t)\|$, we see that $v$ indeed solves the desired ODE. That $\Phi(0,(p,\hat{v}))=(p,\hat{v})$ immediately follows from $e^0=\mathrm{Id}$ and $\|\hat{v}\|=1$.
\end{proof}

\begin{lem} \label{thm:flow-map-is-injective}
    The flow map $\Phi: \R \times S_0 \to S$ defined by
$$ \Phi(t,(p_0,v_0)) = \begin{pmatrix} p_0, \frac{e^{2t g_{p_0}} v_0}{\|e^{2t g_{p_0}} v_0\|} \end{pmatrix}$$
is injective on $\mathbb{R} \times S_0$.
\end{lem}

\begin{proof} 
Suppose $\Phi(t, (p, v)) = \Phi(s, (q, w))$ for two points $(t, (p, v)), (s, (q, w)) \in \mathbb{R} \times S_0$. By definition of $\Phi$, the base points must satisfy $p=q$. We are left with:
\[
\frac{e^{2t g_p} v}{\|e^{2t g_p} v\|} = \frac{e^{2s g_p} w}{\|e^{2s g_p} w\|}.
\]
Thus there exists $\lambda > 0$ such that:
\begin{equation} \label{eq:injectivity-main}
e^{2t g_p} v = \lambda e^{2s g_p} w.
\end{equation}
Rearranging \eqref{eq:injectivity-main} gives:
\begin{equation}\label{eq:injectivity-main-2}
    w = \lambda^{-1} e^{2(t - s) g_p} v.
\end{equation}

To compute the matrix exponential more explicitly we orthogonally diagonalize the symmetric bilinear form $g_p$, finding an orthogonal matrix $U_p$ such that $U_p^T g_p U_p= \eta_{(k,n-k)} $, where $\eta_{(k,n-k)}$ is the standard semi-Riemannian inner product of signature $(k,n-k)$ on $\R^n$. We obtain
$$e^{2(t - s) g_p} = U_p^T \begin{bmatrix}
    e^{-2(t -s)} \, \mathbf{I}_{k \times k} &  \\
     & e^{2(t - s)} \,\mathbf{I}_{(n-k) \times (n-k)}\end{bmatrix} U_p .$$

As in the proof of $S_0$ being a $(2n-2)$-dimensional topological submanifold (cf. \eqref{eq:psiS0}), we note that 
\begin{equation}\label{eq:v'w'inS0}
    v,w\in S_0 \iff v',w'\in \{v'\in \R^n: \|v'\|=1 \text{ and } \eta_{(k,n-k)}(v',v')=0\}\},
\end{equation}
where we define $v':=U_pv$ and $w':=U_pw$, and that $v=w$ if and only if $v'=w'$. We denote the first $k$-components of $v'$ by $v'_-$ and the last $n-k$ components by $v'_+$:
\[
v' =\begin{pmatrix} v'_- \\ v'_+ \end{pmatrix} 
\quad \text{where} \quad
\begin{cases}
v'_- = (v'_1,\dots,v'_k) \in \mathbb{R}^k \\
v'_+ = (v'_{k+1},\dots,v'_{n}) \in \mathbb{R}^{n-k}.
\end{cases}
\]
and analogously for $w'$. Via \eqref{eq:v'w'inS0} the condition of $v,w\in S_0$ becomes equivalent to
\begin{gather} \label{eq:norm-conditions}
\|v_-'\|=\|v'_+\|=\|w'_-\|=\|w'_+\|=\frac{1}{2}. 
\end{gather}

Additionally \eqref{eq:injectivity-main-2} reduces to
\begin{equation}\label{eq:w'}
    w'= \lambda^{-1} \begin{bmatrix}
    e^{-2(t -s)} \, \mathbf{I}_{k \times k} &  \\
     & e^{2(t - s)} \,\mathbf{I}_{(n-k) \times (n-k)}\end{bmatrix} v',
\end{equation}
which splits into 
\[
w'_- = \lambda^{-1} e^{-2(t - s)} v'_- \quad \text{ and } \quad w'_+=\lambda^{-1}  e^{2(t - s)} v'_+.
\]
Taking norms and using \eqref{eq:norm-conditions} we get
$$\lambda^{-1} e^{-2(t - s)}=1=\lambda^{-1}  e^{2(t - s)}  ,$$
which can only hold for $t=s$ and $\lambda=1$. 
Thus $\Phi$ is injective. 
\end{proof}

\end{appendices}


\printbibliography
\addcontentsline{toc}{section}{References}  

\end{document}